\documentclass[a4paper,12pt,reqno]{amsart}

\usepackage{amsmath}
\usepackage{amscd}
\usepackage{amssymb}
\usepackage{mathrsfs}
\usepackage[left=2.5cm,right=2.5cm,bottom=3cm,top=3cm]{geometry}
\newtheorem{thm}{Theorem}[section]

\newtheorem{lem}[thm]{Lemma}

\newtheorem{prop}[thm]{Proposition}
\theoremstyle{definition}
\newtheorem{defn}[thm]{Definition}
\newtheorem{rem}[thm]{Remark}
\newtheorem{exmp}[thm]{Example}
\numberwithin{equation}{section}

\begin{document}

\title{On higher-rank Khovanskii-Teissier inequalities}

\author{Yashan Zhang}

\address{School of Mathematics, Hunan University, Changsha 410082, China}
\email{yashanzh@hnu.edu.cn}
\thanks{Partially supported by Fundamental Research Funds for the Central Universities (No. 531118010468) and National Natural Science Foundation of China (No. 12001179)}

\begin{abstract}
We shall discuss a higher-rank Khovanskii-Teissier inequality, generalizing a theorem of Li in \cite{Li16}. In the course of the proof, we develop new Hodge-Riemann bilinear relations in certain mixed and degenerate settings, which in themselves slightly extend the existing results and imply new Khovanskii-Teissier type inequalities and log-concavity results.
\end{abstract}

\maketitle

\section{Introduction}

\subsection{Backgrounds} Around the year 1979, Khovanskii and Teissier independently discovered deep inequalities in algebraic geometry, which are profound analogs of Alexandrov-Fenchel inequalities in convex geometry. There are many remarkable further developments on Khovanskii-Teissier type inequalities (see e.g. \cite{BFJ,C,Co,De,DN,FX,Gr,La,LX,Li17,RT,Te,X} and references therein), among which we may recall the following one on an $n$-dimensional compact K\"ahler manifold $X$ as an example (see \cite{De,DN,Gr}). If $\omega_1,...,\omega_{n-1}$ are K\"ahler metrics on $X$ and arbitrarily take $[\alpha]\in H^{1,1}(X,\mathbb R)$, then we have
$$\left(\int_X\omega_1\wedge...\wedge\omega_{n-2}\wedge\omega_{n-1}\wedge\alpha\right)^2\ge\left(\int_X\omega_1\wedge...\wedge\omega_{n-2}\wedge\omega_{n-1}^2\right)\left(\int_X\omega_1\wedge...\wedge\omega_{n-2}\wedge\alpha^2\right),$$
and the equality holds if and only if $[\alpha]$ and $[\omega_{n-1}]$ are proportional.

Given the above inequality, as proposed in \cite{Li13,Li16}, it is natural to ask whether some similar inequalities hold for elements in $H^{p,p}(X,\mathbb R)$ for $p\ge2$, which may be called \emph{higher-rank Khovanskii-Teissier inequalities} (see \cite{RT}). \\

In the remaining part of this note, $X$ is an $n$-dimensional compact K\"ahler manifold with a fixed background K\"ahler metric $\omega_X$, and $h^{p,q}:=dim H^{p,q}(X,\mathbb C)$ is the Hodge number of $X$.

Our study here is mainly motivated by Li's works \cite{Li13,Li16}, which gave a very natural higher-rank generalization of the Khovanskii-Teissier inequality recalled above. More precisely, given $2\le p\le[n/2]$ and K\"ahler metrics $\omega,\omega_1,...,\omega_{n-2p}$ on $X$, and denote $\Omega:=\omega_1\wedge...\wedge\omega_{n-2p}$, then \cite[Theorem 1.3(3)]{Li16} states that the followings are equivalent:
\begin{itemize}
\item[(i)] For every $[\gamma]\in H^{p,p}(X,\mathbb C)$, there holds
\begin{align}\label{kt0}
\left(\int_X\Omega\wedge\omega^{p}\wedge\gamma\right)\left(\int_X\Omega\wedge\omega^{p}\wedge\bar\gamma\right)\ge\left(\int_X\Omega\wedge\omega^{2p}\right)\left(\int_X\Omega\wedge\gamma\wedge\bar\gamma\right);
\end{align}
\item[(ii)] For all $1\le l\le[p/2]$, $h^{2l-1,2l-1}=h^{2l,2l}$;
\end{itemize} 
moreover, if condition (ii) holds, then a $[\gamma]\in H^{p,p}(X,\mathbb C)$ satisfies the equality in \eqref{kt0} if and only if $[\gamma]$ is proportional to $[\omega^{p}]$. \\

One can find more recent  progresses on higher-rank Khovanskii-Teissier inequalities on certain Schur classes of ample vector bundles in Ross-Toma's work \cite{RT}. \\

\subsection{A generalized $m$-positivity and statements of results} In this note, we shall extend Li's above-recalled result to a more general setting, in which the involved forms/classes are no longer required to be positive. To state the results, let's firstly give some necessary preparations, particularly including a notion called generalized $m$-positivity. 

\subsubsection{A generalized $m$-positivity} We first introduce the positivity condition used in the discussions.
\begin{defn}\label{m-pos}
Let $\Phi$ be a (strictly) positive $(m,m)$-form on $X$ and $\phi$ a real $(1,1)$-form on $X$. We say a real $(1,1)$-form $\alpha$ on $X$ is \emph{$(n-m)$-positive with respect to $(\Phi,\phi)$} if 
$$\Phi\wedge\phi^{n-m-k}\wedge\alpha^k>0$$	
for any $1\le k\le n-m$. In particular, the case that $(\Phi,\phi)=(\omega_X^{m},\omega_X)$ gives the original $(n-m)$-positivity with respect to a fixed K\"ahler metric $\omega_X$, and in this case we say $\alpha$ is $(n-m)$-positive with respect to $\omega_X$.
\end{defn}
It will be seen from the following discussions (particularly Lemma \ref{lem_c}) that the above generalized $m$-positivity notion naturally appears as the characterization of the G{\aa}rding cone of certain hyperbolic polynomial. Moreover, comparing with the original $m$-positivity with respect to one fixed K\"ahler metric, one of the \emph{important advantages} of the above generalized $m$-positivity is that \emph{it is very flexible when the induction or iteration arguments are involved}. We should mention that a version of $(n-m)$-positivity with respect to $m$ K\"ahler metrics $\omega_1,...,\omega_{m}$ has been proposed in \cite[Remark 2.12]{X}, which seems slightly stronger than the one defined above.

\subsubsection{Higher-rank Khovanskii-Teissier inequalities} The following is our first result, which extends previous results of Li \cite{Li13,Li16} to a degenerate setting.
\begin{thm}\label{thm_kt}
Assume $\lambda_0=0$ and $\lambda_1,...,\lambda_N\in\mathbb Z_{\ge1}$ with $\sum_{i=1}^N\lambda_i=n-2p$, $p\in\mathbb Z_{\ge1}$. Assume $\alpha_{ij_i}, 1\le i\le N$ and $1\le j_i\le\lambda_i$, be semi-positive closed real $(1,1)$-forms on $X$ such that $\alpha_{ij_i}$ has at least $(n-\sum_{s=0}^{i-1}\lambda_s)$ positive eigenvalues. Denote $\Omega:=\bigwedge_{1\le i\le N,1\le j_i\le \lambda_i}\alpha_{ij_i}$. Assume $\eta$ is a semi-positive closed real $(1,1)$-forms on $X$ of at least $(n-\sum_{s=1}^{N}\lambda_s)$ positive eigenvalues, and $\alpha$ a closed real $(1,1)$-forms on $X$ which is $2$-positive with respect to $(\Omega\wedge\eta^{n-2-\sum_{s=1}^{N}\lambda_s},\eta)$. Assume $p\ge2$. Then the followings are equivalent:
\begin{itemize}
\item[(i)] For every $[\gamma]\in H^{p,p}(X,\mathbb C)$, there holds
\begin{align}\label{kt}
\left(\int_X\Omega\wedge(\eta^{p-1}\wedge\alpha)\wedge\gamma\right)\left(\int_X\Omega\wedge(\eta^{p-1}\wedge\alpha)\wedge\bar\gamma\right)\ge\left(\int_X\Omega\wedge(\eta^{p-1}\wedge\alpha)^2\right)\left(\int_X\Omega\wedge\gamma\wedge\bar\gamma\right);
\end{align}
\item[(ii)] For all $1\le l\le[p/2]$, $h^{2l-1,2l-1}=h^{2l,2l}$.
\end{itemize} Moreover, if condition (ii) holds, then a $[\gamma]\in H^{p,p}(X,\mathbb C)$ satisfies the equality in \eqref{kt} if and only if $[\gamma]$ is proportional to $[\eta^{p-1}\wedge\alpha]$. 
\end{thm}

\begin{rem}
(1) We may point out that the above involved $(1,1)$-forms $\alpha_{ij_i},i\ge2$, $\eta$ and $\alpha$ could be degenerate in some directions, while $\alpha$ could be negative in some directions; moreover, the $(p,p)$-form $\eta^{p-1}\wedge\alpha$ is partially mixed/polarized (compare with \eqref{kt0}).\\
(2) The condition (ii) in Theorem \ref{thm_kt} is satisfied by a variety of examples, see \cite[Example 1.7]{Li13}.
\end{rem}

\begin{exmp}
Let's look at a special case of Theorem \ref{thm_kt} that $p=2$ and $h^{2,2}=h^{1,1}$. Assume there exists a holomorphic submersion $f:X\to Y$ with $Y$ a $4$-dimensional compact K\"ahler manifold. Let $\omega_1,...,\omega_{n-4}$ be K\"ahler metrics on $X$, $\chi$ a K\"ahler metric on $Y$ and $\beta$ a closed real $(1,1)$-form on $Y$ which is $2$-positive with respect to $\chi$ on $Y$. Denote $\Omega:=\omega_1\wedge...\wedge\omega_{n-4}$.
Then for any $[\gamma]\in H^{2,2}(X,\mathbb C)$ we have
\begin{align}
\left(\int_X\Omega\wedge f^*\chi\wedge f^*\beta\wedge\gamma\right)\left(\int_X\Omega\wedge f^*\chi\wedge f^*\beta\wedge\bar\gamma\right)\ge\left(\int_X\Omega\wedge f^*\chi^2\wedge f^*\beta^2\right)\left(\int_X\Omega\wedge\gamma\wedge\bar\gamma\right)\nonumber.
\end{align}
and the equality holds if and only if $[\gamma]$ is proportional to $[f^*\chi\wedge f^*\beta]$.
\end{exmp}

Similar to \cite{Li16}, our Theorem \ref{thm_kt} will be proved by making use of the Hodge-Riemann bilinear relations. To this end, however, we have to first develop the Hodge-Riemann bilinear relations in the corresponding mixed and degenerate setting, which we now intruduce as follows. 

\subsubsection{Hodge index theorem}
Let's begin with a particular piece of Hodge-Riemann bilinear relation, namely, the Hodge index theorem (i.e. Hodge-Riemann bilinear relation on $H^{1,1}$ level).

\begin{defn}(\cite[Section 4]{DN})\label{defn_hi}
For any $[\Omega]\in H^{n-2,n-2}(X,\mathbb R):=H^{n-2,n-2}(X,\mathbb C)\cap H^{2n-4}(X,\mathbb R)$ and $[\eta]\in H^{1,1}(X,\mathbb R)$, we define the \emph{primitive space with respect to $([\Omega],[\eta])$} by
$$P^{1,1}(X,\mathbb C):=\left\{[\gamma]\in H^{1,1}(X,\mathbb C)|[\Omega]\wedge[\eta]\wedge[\gamma]=0\right\}.$$
Then we say \emph{$([\Omega],[\eta])$ satisfies the Hodge index theorem} if the quadratic form
$$Q([\beta],[\gamma]):=\int_X\Omega\wedge\beta\wedge\overline\gamma$$
is negative definite on $P^{1,1}(X,\mathbb C)$. 
\par Let $\mathscr {H}$ be the set of pair $([\Omega],[\eta])\in H^{n-2,n-2}(X,\mathbb R)\times H^{1,1}(X,\mathbb R)$ satisfying the Hodge index theorem.
\end{defn}

A fundamental question is to determine $\mathscr H$.

\par The classical Hodge index theorem (see e.g. \cite[Chapter 6]{V}) equivalently says that $([\omega^{n-2}],[\omega])\in\mathscr H$ for any K\"ahler metric $\omega$. A theorem of Gromov \cite{Gr} (also see \cite[Theorem 1.2]{DN}) extended Hodge index theorem to a mixed setting, proving that $([\omega_1\wedge...\wedge\omega_{n-2}],[\omega_{n-1}])\in\mathscr H$ for any K\"ahler metrics $\omega_1,...,\omega_{n-1}$. More recently, Xiao \cite[Theorem A]{X} further proved that if $\omega$ is a K\"ahler metric and $\alpha_1,...,\alpha_{n-m-1}$ are closed real $(1,1)$-forms such that every $\alpha_j$ is $(n-m)$-positive ($n-m\ge2$) with respect to $\omega$, then $([\omega^{m}\wedge\alpha_1\wedge...\wedge\alpha_{n-m-2}],[\alpha_{n-m-1}])\in\mathscr H$. Moreover, Ross-Toma \cite{RT} proved that certain Schur classes of ample vector bundle are contained in $\mathscr H$.

Comparing Xiao's result \cite{X} with the above-mentioned theorem of Gromov \cite{Gr}, it seems very natural to ask that can we further mix the term $\omega^{m}$ in Xiao's result, e.g. can we replace it by $\omega_1\wedge...\wedge\omega_{m}$? We shall answer this in the following

\begin{thm}\label{thm1}
Fix an integer $m\le n-2$. Assume $\omega_1,...,\omega_{m}$ are K\"ahler metrics on $X$, and $\alpha_1,...,\alpha_{n-m-1}$ closed real $(1,1)$-forms on $X$ such that every $\alpha_j$ is $(n-m)$-positive with respect to $(\omega_1\wedge...\wedge\omega_{m},\omega_X)$. Then 
$$([\omega_1\wedge...\wedge\omega_{m}\wedge\alpha_1\wedge...\wedge\alpha_{n-m-2}],[\alpha_{n-m-1}])\in\mathscr H.$$
\end{thm}
 
Indeed, we shall prove the following more general Hodge index theorem in a setting related to Theorem \ref{thm_kt}, which contains Theorem \ref{thm1} as a very special case and will be used in the proof of the main result Theorem \ref{thm_kt}.
\begin{thm}\label{thm2}
Assume $\lambda_1,...,\lambda_N,\lambda_{N+1}\in\mathbb Z_{\ge1}$ with $\sum_{i=1}^{N+1}\lambda_s= n-2$. Assume $\alpha_{ij_i}, 1\le i\le N+1$ and $1\le j_i\le\lambda_i+1$, are closed real $(1,1)$-forms on $X$ such that $\alpha_{1j_1}$'s are K\"ahler metrics, and when $i\ge2$, $\alpha_{ij_i}$ is $(n-\sum_{s=1}^{i-1}\lambda_s)$-positive with respect to\\ $(\bigwedge_{1\le s\le i-1,1\le j_s\le \lambda_s}\alpha_{sj_s},\alpha_{i-1,\lambda_{i-1}+1})$. Denote $\Omega:=\bigwedge_{1\le i\le N+1,1\le j_i\le \lambda_i}\alpha_{ij_i}$. Assume $\eta$ is a closed real $(1,1)$-forms on $X$ which is $2$-positive with respect to  $(\Omega,\alpha_{N+1,\lambda_{N+1}+1})$. Then we have
$$([\Omega],[\eta])\in\mathscr H.$$
\end{thm}

\begin{rem}
We should point out that, in Theorem \ref{thm2}, the case that $\alpha_{11}=...=\alpha_{1,\lambda_1}=:\omega_X$ and all $\alpha_{ij_i}$'s, $2\le i\le N+1$ and $1\le j_i\le\lambda_{i}$, and $\eta$ are $(n-\lambda_1)$-positive with respect to $\omega_X$ is exactly Xiao's result \cite[Theorem A]{X}. Our result here extends it to a more mixed setting and partially weakens the positivity assumption simultaneously, in which we may in particular mention that some factor $\alpha_{ij_i}$ is \emph{no longer} required to have at least $(n-\lambda_1)$ positive eigenvalues, e.g. they may only have $3$ positive eigenvalues, while $\eta$ may only have $2$ positive eigenvalues.
\end{rem}

To prove Theorem \ref{thm2}, we will mainly follow Xiao's strategy in \cite{X}. One of the crucial steps in \cite{X} is applying G{\aa}rding inequality to the elementary symmetric polynomial. In this note, we shall consider some hyperbolic polynomials which are more general than the elementary symmetric polynomial, and then apply G{\aa}rding's theory \emph{iteratively}. Details will be presented in Section \ref{sect_hi}.

\subsubsection{Hodge-Riemann bilinear relations}\label{sect_hr}
Our Theorem \ref{thm2} motivates some general Hodge-Riemann bilinear relations in the mixed settings. Similar to Definition \ref{defn_hi} we first introduce
\begin{defn}(\cite[Section 4]{DN})\label{defn_hr}
For any $[\Omega]\in H^{n-k,n-k}(X,\mathbb R):=H^{n-k,n-k}(X,\mathbb C)\cap H^{2(n-k)}(X,\mathbb R)$, $[\eta]\in H^{1,1}(X,\mathbb R)$, and for $(p,q)$ with $0\le p,q\le p+q\le n$ and $p+q=k$, we define the \emph{primitive space with respect to $([\Omega],[\eta])$} in $H^{p,q}(X,\mathbb C)$ by
$$P^{p,q}(X,\mathbb C)=P^{p,q}_{([\Omega],[\eta])}(X,\mathbb C):=\left\{[\gamma]\in H^{p,q}(X,\mathbb C)|[\Omega]\wedge[\eta]\wedge[\gamma]=0\right\}.$$
Then we say \emph{$([\Omega],[\eta])$ satisfies the Hodge-Riemann bilinear relation on $H^{p,q}(X,\mathbb C)$} if the quadratic form
$$Q([\beta],[\gamma]):=\sqrt{-1}^{q-p}(-1)^{\frac{(p+q)(p+q+1)}{2}}\int_X\Omega\wedge\beta\wedge\overline\gamma$$
is positive definite on $P^{p,q}(X,\mathbb C)$.
\par Let $\mathscr H^{n-k,n-k}$ be the set of pair $([\Omega],[\eta])\in H^{n-k,n-k}(X,\mathbb R)\times H^{1,1}(X,\mathbb R)$ satisfying the Hodge-Riemann bilinear relation on $H^{p,q}(X,\mathbb C)$ for any $(p,q)$ with $p+q=k$.
\end{defn}

\par The classical Hodge-Riemann bilinear relation (see e.g. \cite[Chapter 6]{V}) equivalently says $([\omega^{n-k}],[\omega])\in\mathscr H^{n-k,n-k}$ for any K\"ahler metric $\omega$. A theorem of Dinh-Nguy\^en \cite[Theorem A]{DN} and Cattani \cite{C} extended Hodge-Riemann bilinear relation to a mixed setting, proving that $([\omega_1\wedge...\wedge\omega_{n-k}],[\omega_{n-k+1}])\in\mathscr H^{n-k,n-k}$ for any K\"ahler metrics $\omega_1,...,\omega_{n-k+1}$ (also see \cite{Gr,T} for some special cases). More recently, Xiao \cite[Theorem A, Remarks 2.2 and 3.9]{X2} further proved that for any fixed positive integer $m$ with $k\le m\le n$, if $\omega_1,...,\omega_{n-m}$ are K\"ahler metrics on $X$ and $\alpha_1,...,\alpha_{m-k+1}$ are closed real $(1,1)$-forms on $X$ such that every $\alpha_j$ is semi-positive and of at least $m$ positive eigenvalues, then $([\omega_1\wedge...\wedge\omega_{n-m}\wedge\alpha_1\wedge...\wedge\alpha_{m-k}],[\alpha_{m-k+1}])\in\mathscr H^{n-k,n-k}$. 

Motivated by Theorem \ref{thm2}, we shall extend Xiao's result \cite[Theorem A, Remarks 2.2 and 3.9]{X2} to the following

\begin{thm}\label{thm3}
Given $N\in\mathbb Z_{\ge0}$. Assume $\lambda_0=0$ and $\lambda_1,...,\lambda_N\in\mathbb Z_{\ge1}$ with $\sum_{i=0}^N\lambda_i=n-(p+q)$, $p,q\in\mathbb Z_{\ge1}$. For $1\le i\le N$, assume $\alpha_{ij_i}$, $1\le j_i\le\lambda_i$, are semi-positive closed real $(1,1)$-forms on $X$ such that $\alpha_{ij_i}$ has at least $(n-\sum_{s=0}^{i-1}\lambda_s)$ positive eigenvalues. Denote $\Omega:=\bigwedge_{1\le i\le N,1\le j_i\le \lambda_i}\alpha_{ij_i}$ if $N\ge1$ and $\Omega=1\in\mathbb R$ if $N=0$. Assume $\eta_1,...,\eta_{p+q-2},\eta_{p+q-1}$ are semi-positive closed real $(1,1)$-forms on $X$ of at least $(n-\sum_{s=0}^{N}\lambda_s)$ positive eigenvalues (note that $p+q-2=n-2-\sum_{i=1}^N\lambda_s$). Then for $0\le d\le[(p+q)/2]-1$,
$$([\Omega\wedge\eta_1\wedge...\wedge\eta_{2d}],[\eta_{2d+1}])\in\mathscr H^{n-p-q+2d,n-p-q+2d}.$$
\end{thm}

\begin{rem}\label{rem_N=0}
Note that in Theorem \ref{thm3}, the $N=0$ case is exactly the mixed Hodge-Riemann bilinear relations of Dinh-Nguy\^en \cite[Theorem A]{DN} and Cattani \cite{C}, as in this case $\Omega$ is the constant $1\in\mathbb R$ and $\eta_i$'s are all K\"ahler, while the $N=1$ case is Xiao's result \cite[Theorem A, Remarks 2.2 and 3.9]{X2}.
\end{rem}

\begin{rem}\label{rem_hr}
Consequently, by standard arguments (see e.g. \cite{DN,DN13,T,X2}) we also have the Hard Lefschetz Theorem and Lefschetz Decomposition Theorem (with respect to $([\Omega\wedge\eta_1\wedge...\wedge\eta_{2d}],[\eta_{2d+1}])$) as follows.
\par Assume the same setting and notations as in Theorem \ref{thm3}, then
\begin{itemize}
\item[(a)] The map $[\Omega\wedge\eta_1\wedge...\wedge\eta_{2d}]:H^{p-d,q-d}(X,\mathbb C)\to H^{n-q+d,n-p+d}(X,\mathbb C)$ is an isomorphism;
\item[(b)] The space $H^{p-d,q-d}(X,\mathbb C)$ has a $Q$-orthogonal direct sum decomposition
$$H^{p-d,q-d}(X,\mathbb C)=P^{p-d,q-d}(X,\mathbb C)\oplus[\eta_{2d+1}]\wedge H^{p-d-1,q-d-1}(X,\mathbb C).$$
\item[(c)] $dim P^{p-d,q-d}(X,\mathbb C)=h^{p-d,q-d}-h^{p-d-1,q-d-1}$.
\end{itemize}
\end{rem}

\begin{rem}
In Theorem \ref{thm3}, the case that $p=q=1$ and $d=0$ gives a Hodge index theorem, which is \emph{weaker} than the one proved in Theorem \ref{thm2}, as being $k$-positive (in the general sense of Definition \ref{m-pos}) is more general than being semi-positive and of at least $k$ positive eigenvalues. The better one in Theorem \ref{thm2} is the reason that the $\alpha$ in the main Theorem \ref{thm_kt} can be assumed to be only $2$-positive.
\end{rem}

\section{Hodge index Theorem: Proof of Theorem \ref{thm2}}\label{sect_hi}

\subsection{Hyperbolic polynomials} For the proof of Theorem \ref{thm2}, we first collect some useful properties of hyperbolic polynomials, closely following \cite{Ga}, \cite[Chapter 2]{H} and \cite[Section 2.1]{X}.
\par Let $V$ be an $N$-dimensional complex vector space and $P=P(x)$ a homogeneous polynomial of degree $n$ on $V$. For a real vector $a\in V$ we say that $P$ is \emph{hyperbolic at $a$} if the equation $P(sa+x)=0$ has $n$ real zeroes for every real $x\in V$. In particular, $P(a)\neq0$ and the polynomial $P(x)/P(a)$ is real when $x$ is real. Then we may without loss of generality assume that $P$ is a real polynomial. We say $P$ is \emph{complete} if the condition $P(sx+y)=P(y)$ for all $s,y$ implies $x=0$. If $P$ is hyperbolic at $a$, let $C(P,a)$ be the set of all $x$ such that $P(sa+x)\neq0$ when $s\ge0$, and we call it the \emph{G{\aa}rding cone} of $P$.
\begin{thm}\cite{Ga,H}\label{thm2.3}
If $P$ is hyperbolic at $a$, then\\
(1) $C(P,a)$ is an open convex cone;\\
(2) For any $b\in C(P,a)$,  $P$ is hyperbolic at $b$ and $C(P,b)=C(P,a)$. In this case we denote $C(P)=C(P,a)$;\\
(3) $P(x)/P(a)>0$ on $C(P,a)$ and $\left(P(x)/P(a)\right)^{1/n}$ is concave on $C(P,a)$.
\end{thm}
Let $\tilde P(x_1,...,x_n)$ be the completely polarized form of $P$, which can be explicitly obtained by differentiation:
$$\tilde P(x_1,...,x_n)=\frac{1}{n!}\prod_{k=1}^n\left(\sum_{j=1}^Nx^j_k\frac{\partial}{\partial x^j}\right)P(x),$$
where $x=(x^1,...,x^N)$ and $x_k=(x_k^1,...,x_k^N)$.

\begin{thm}\cite{Ga,H}\label{thm2.4}
If $P$ is hyperbolic at $a$ with $P(a)>0$ and complete, then \\
(1) For any $2\le m\le n-1$ and any $b_1,...,b_{n-m}\in C(P,a)$, $P_{m}(x)=\tilde P(b_1,...,b_{n-m},x,...,x)$ is hyperbolic at $a$ and complete, and $C(P,a)\subset C(P_{m},a)$. In particular, $\tilde P(x_1,x_2,...,x_n)>0$ for any $x_1,...,x_{n-1}\in C(P,a)$ and $x_n\in \overline{C(P,a)}\setminus\{0\}$;\\
(2) $\tilde P(x_1,x_2,...,x_n)\ge P(x_1)^{1/n}\cdot...\cdot P(x_n)^{1/n}$ for any $x_1,...,x_n\in C(P,a)$, and the equality holds if and only if $x_j$'s are pairwise proportional.
\end{thm}

\subsection{Proof}\label{hi_pf}
We now prove Theorem \ref{thm2}. As in \cite{X} and \cite{T,DN}, we shall first deal with the linear setting on $\mathbb C^n$. However, unlike \cite{X} (in which the discussions began with the elementary symmetric polynomial $\sigma_m$), here we will begin with the determinant function defined on $n\times n$ Hermitian matrices (see e.g. \cite[Example 4]{Ga}), which provides us some more freedoms to mix the arguments involved. 

That is equivalent to consider the real $(1,1)$-forms of constant coefficients on $\mathbb C^n$ which we do now. Precisely, if we denote $\Lambda^{1,1}_{\mathbb R}(\mathbb C^n)$ be the space of real $(1,1)$-forms of constant coefficients on $\mathbb C^n$, then for any $x\in\Lambda^{1,1}_{\mathbb R}(\mathbb C^n)$ we define
$$P(x):=x^n=x\wedge...\wedge x,$$
which is a homogenous polynomial of degree $n$ and is hyperbolic at any positive real $(1,1)$-forms in $\Lambda^{1,1}_{\mathbb R}(\mathbb C^n)$ and is complete. Precisely, for any fixed positive real $(1,1)$-forms $\omega\in\Lambda^{1,1}_{\mathbb R}(\mathbb C^n)$, the cone $C(P,\omega)$ consists of all positive real $(1,1)$-forms in $\Lambda^{1,1}_{\mathbb R}(\mathbb C^n)$. Note that the completely polarized form of $P$ is given by
$$\tilde P(x_1,...,x_n):=x_1\wedge...\wedge x_n$$
for any $x_1,...,x_n\in\Lambda^{1,1}_{\mathbb R}(\mathbb C^n)$. Fix a K\"ahler metric $\omega_0\in\Lambda^{1,1}_{\mathbb R}(\mathbb C^n)$. Then for any fixed integer $\lambda_1\le n-2$ and any K\"ahler metrics $\alpha_{11},...,\alpha_{1\lambda_1}\in\Lambda^{1,1}_{\mathbb R}(\mathbb C^n)$, the polynomial 
\begin{equation}\label{p_m}
P_{1}(x):=\tilde P(\alpha_{11},...,\alpha_{1\lambda_1},x,...,x)=\alpha_{11}\wedge...\wedge\alpha_{1\lambda_1}\wedge x^{n-\lambda_1}
\end{equation}
is hyperbolic at $\omega_0$ and complete, thanks to Theorem \ref{thm2.4}(1). 

Denote $\alpha_{1,\lambda_1+1}:=\omega_0$. Next we need to understand the cone $C(P_{1},\alpha_{1,\lambda_1+1})$. By Theorem \ref{thm2.4}(1) we know that $C(P,\alpha_{1,\lambda_1+1})\subset C(P_{1},\alpha_{1,\lambda_1+1})$, i.e. $C(P_{1},\alpha_{1,\lambda_1+1})$ contains all positive real $(1,1)$-forms in $\Lambda^{1,1}_{\mathbb R}(\mathbb C^n)$. More generally, we have
\begin{lem}\label{lem_c}
$C(P_{1},\alpha_{1,\lambda_1+1})=\{x\in\Lambda^{1,1}_{\mathbb R}(\mathbb C^n)|$$x$ is $(n-\lambda_1)$-positive with respect to $(\alpha_{11}\wedge...\wedge\alpha_{1\lambda_1},\alpha_{1,\lambda_1+1})\}$.
\end{lem}
\begin{proof}
For convenience, set $C_1:=\{x\in\Lambda^{1,1}_{\mathbb R}(\mathbb C^n)|$$x$ is $(n-\lambda_1)$-positive with respect to $(\alpha_{1,\lambda_1+1},\alpha_{11}\wedge...\wedge\alpha_{1\lambda_1})\}$, and $\Phi:=\alpha_{11}\wedge...\wedge\alpha_{1\lambda_1}$. In this case, since $P_{1}(\alpha_{1,\lambda_1+1})>0$, it is easy to see that $C(P_{1},\alpha_{1,\lambda_1+1})$ consists of all $x\in\Lambda^{1,1}_{\mathbb R}(\mathbb C^n)$ such that $P_1(s\alpha_{1,\lambda_1+1}+x)>0$ whenever $s\ge0$. On the other hand, by definition \eqref{p_m},
\begin{align}\label{eq_p_1}
P_1(s\alpha_{1,\lambda_1+1}+x)&=\Phi\wedge(s\alpha_{1,\lambda_1+1}+x)^{n-\lambda_1}\nonumber\\
&=\sum_{k=0}^{n-\lambda_1}\left(\binom{n-\lambda_1}{k}\Phi\wedge\alpha_{1,\lambda_1+1}^{n-\lambda_1-k}\wedge x^k\right)s^{n-\lambda_1-k}\nonumber\\
&=\sum_{k=0}^{n-\lambda_1}\mu_{k}\cdot s^{n-\lambda_1-k},
\end{align}
where we have denoted
\begin{equation}\label{eq_mu_k}
\mu_{k}:=\binom{n-\lambda_1}{k}\Phi\wedge\alpha_{1,\lambda_1+1}^{n-\lambda_1-k}\wedge x^k.
\end{equation}
Using \eqref{eq_p_1} one easily sees that if $x\in C_1$, i.e. $x$ is $(n-\lambda_1)$-positive with respect to $(\alpha_{1,\lambda_1+1},\Phi)$, then every coefficient $\mu_{k}$, $0\le k\le n-\lambda_1$, of $P_1(s\alpha_{1,\lambda_1+1}+x)$ (as a polynomial of $s$) is positive and hence $P_1(s\alpha_{1,\lambda_1+1}+x)>0$ whenever $s\ge0$, i.e. $x\in C(P_{1},\alpha_{1,\lambda_1+1})$. This proves $C_1\subset C(P_{1},\alpha_{1,\lambda_1+1})$.

To see the converse, we assume $x\in C(P_{1},\alpha_{1,\lambda_1+1})$, i.e. 
$$h(s):=P_1(s\alpha_{1,\lambda_1+1}+x)=\sum_{k=0}^{n-\lambda_1}\mu_{k}\cdot s^{n-\lambda_1-k}$$ 
has $(n-\lambda_1)$ negative zeroes. In particular, $\mu_{n-\lambda_1}=h(0)>0$. To proceed we take the derivative of $h(s)$,
$$h'(s)=\sum_{k=0}^{n-\lambda_1-1}(n-\lambda_1-k)\mu_{k}\cdot s^{n-\lambda_1-1-k},$$
which by Rolle's Theorem has $(n-\lambda_1-1)$ negative zeroes. In particular, $\mu_{n-\lambda_1-1}=h'(0)>0$. Iterating this argument, we finally conclude that $\mu_k>0$ for every $1\le k\le n-\lambda_1$, which, by the definition of $\mu_k$ in \eqref{eq_mu_k}, means $x\in C_1$. This proves $ C(P_{1},\alpha_{1,\lambda_1+1})\subset C_1$.
\end{proof}

\begin{rem}
For any other K\"ahler metric $\tilde\omega$ on $X$, by Theorem \ref{thm2.3}(2) we know $C(P_1,\tilde\omega)=C(P_1,\alpha_{1,\lambda_1+1})$. Therefore, as a consequence of  Lemma \ref{lem_c}, $x\in\Lambda^{1,1}_{\mathbb R}(\mathbb C^n)$ is $(n-\lambda_1)$-positive with respect to $(\alpha_{11}\wedge...\wedge\alpha_{1\lambda_1},\alpha_{1,\lambda_1+1})$ if and only if $x$ is $(n-\lambda_1)$-positive with respect to $(\alpha_{11}\wedge...\wedge\alpha_{1\lambda_1},\tilde\omega)$.
\end{rem}

Given Lemma \ref{lem_c}, we may apply Theorem \ref{thm2.4}(1) one more time to conclude that, for any real $(1,1)$-forms $\alpha_{21},\alpha_{22},...,\alpha_{2\lambda_2},\alpha_{2,\lambda_2+1}\in\Lambda^{1,1}_{\mathbb R}(\mathbb C^n)$ such that every $\alpha_{2j}$ is $(n-\lambda_1)$-positive with respect to $(\alpha_{11}\wedge...\wedge\alpha_{1\lambda_1},\alpha_{1,\lambda_1+1})$, the polynomial
\begin{align}\label{p_md}
P_{2}(x):&=\tilde P_1(\alpha_{21},\alpha_{22},...,\alpha_{2\lambda_2},x,...,x)\nonumber\\
&=\alpha_{11}\wedge...\wedge\alpha_{1\lambda_1}\wedge\alpha_{21}\wedge...\wedge\alpha_{2\lambda_2}\wedge x^{n-\lambda_1-\lambda_2}
\end{align}
is hyperbolic at $\alpha_{2,\lambda_2+1}$ and complete.\\

Similar to Lemma \ref{lem_c}, we have
\begin{lem}\label{lem_c1}
$C(P_{2},\alpha_{2,\lambda_2+1})=\{x\in\Lambda^{1,1}_{\mathbb R}(\mathbb C^n)|$$x$ is $(n-\lambda_1-\lambda_2)$-positive with respect to $(\alpha_{11}\wedge...\wedge\alpha_{1\lambda_1}\wedge\alpha_{21}\wedge...\wedge\alpha_{2\lambda_2},\alpha_{2,\lambda_2+1})\}$.
\end{lem}

\begin{proof}
This follows by the same arguments in Lemma \ref{lem_c}.
\end{proof}

Working \emph{iteratively}, we easily arrive at the following

\begin{lem}\label{lem_c2}
Assume $\lambda_1,...,\lambda_N,\lambda_{N+1}\in\mathbb Z_{\ge1}$ with $\sum_{i=1}^{N+1}\lambda_s= n-2$. Assume $\alpha_{ij_i}\in\Lambda_{\mathbb R}^{1,1}(\mathbb C^n), 1\le i\le N+1$ and $1\le j_i\le\lambda_i+1$, be closed real $(1,1)$-forms on $\mathbb C^n$ such that $\alpha_{1j_1}$'s are K\"ahler metrics, and when $i\ge2$, $\alpha_{ij_i}$ is $(n-\sum_{s=1}^{i-1}\lambda_s)$-positive with respect to $(\bigwedge_{1\le s\le i-1,1\le j_s\le \lambda_s}\alpha_{sj_s},\alpha_{i-1,\lambda_{i-1}+1})$. Denote $\Omega:=\bigwedge_{1\le i\le N+1,1\le j_i\le \lambda_i}\alpha_{ij_i}$. Then 
\begin{itemize}
\item[(1)] the polynomial $P_{N+1}:=\Omega\wedge x^2$ is hyperbolic at $\alpha_{N+1,\lambda_{N+1}+1}$ and complete;
\item[(2)] $C(P_{N+1},\alpha_{N+1,\lambda_{N+1}+1})=\{x\in\Lambda^{1,1}_{\mathbb R}(\mathbb C^n)|$$x$ is $2$-positive with respect to $(\Omega,\alpha_{N+1,\lambda_{N+1}+1})\}$.
\end{itemize}
\end{lem}

Having the above preparations, similar to \cite[Lemmas 3.8, 3.9]{X} we have
\begin{lem}\label{lem_c3}
Assume the same setting and notions as in Lemma \ref{lem_c2}, and let $\eta\in\Lambda_{\mathbb R}^{1,1}(\mathbb C^n)$ be a real $(1,1)$-forms on $\mathbb C^n$ which is $2$-positive with respect to  $(\Omega,\alpha_{N+1,\lambda_{N+1}+1})$.
\begin{itemize}
\item[(1)] The $(n-1,n-1)$-form $\Omega\wedge\eta$ is  strictly positive on $\mathbb C^n$.
\item[(2)] If $x\in\Lambda^{1,1}_{\mathbb R}(\mathbb C^n)$ satisfies $\Omega\wedge\eta\wedge x=0$, then 
$$\Omega\wedge x^2\le0,$$
and the equality holds if and only if $x=0$.
\end{itemize}
\end{lem}
\begin{proof}
Item (1) is a consequence of Theorem \ref{thm2.4}(1). To see this we first note that the given $\eta\in C(P_{N+1},\alpha_{N+1,\lambda_{N+1}+1})$ by Lemma \ref{lem_c2}(2). Moreover, for any non-zero semi-positive real $(1,1)$-form $x_0\in\Lambda^{1,1}_{\mathbb R}(\mathbb C^n)$, we have $x_0\in  \overline{C(P_{N+1},\alpha_{N+1,\lambda_{N+1}+1})}\setminus\{0\}$, as $x_0$ belongs to the closure of the set $\mathcal C\subset \Lambda^{1,1}_{\mathbb R}(\mathbb C^n)$ of all K\"ahler metrics and $\mathcal C\subset C(P_{1},\omega_0)\subset C(P_{N+1},\omega_0)=C(P_{N+1},\alpha_{N+1,\lambda_{N+1}+1})$ by Theorem \ref{thm2.4}(1). Therefore, again by Theorem \ref{thm2.4}(1),
\begin{align}
\Omega\wedge\eta\wedge x_0=P_{N+1}(\eta,x_0)>0.
\end{align}

Item (2) is a consequence of G{\aa}rding inequality in Theorem \ref{thm2.4}(2). Let's consider $P_{N+1}$, which, by Lemma \ref{lem_c2}(1), is hyperbolic at $\alpha_{N+1,\lambda_{N+1}+1}$ and complete.
Since $\eta\in C(P_{N+1},\alpha_{N+1,\lambda_{N+1}+1})$ and $C(P_{N+1},\alpha_{N+1,\lambda_{N+1}+1})$ is an open convex cone in $\Lambda^{1,1}_{\mathbb R}(\mathbb C^n)$ by Theorem \ref{thm2.3}(1), we know that, for sufficiently large number $s$, $x+s\eta\in C(P_{N+1},\alpha_{N+1,\lambda_{N+1}+1})$. Then by Theorem \ref{thm2.4}(2) there holds
$$\tilde P_{N+1}(\eta,x+s\eta)^2\ge P_{N+1}(\eta)\cdot P_{N+1}(x+s\eta),$$
which, as can be checked directly, is equivalent to
$$\tilde P_{N+1}(\eta,x)^2\ge P_{N+1}(\eta)\cdot P_{N+1}(x),$$
from which the desired inequality follows, since by assumption $\tilde P_{N+1}(\eta,x)=0$ and by Theorem \ref{thm2.4}(2) (or the above item (1)) $P_{N+1}(\eta)>0$. Moreover, the equality holds if and only if $\eta$ and $x$ are proportional, which, by using again $P_{N+1}(\eta,x)=0$ and $P_{N+1}(\eta)>0$, in turn implies $x=0$.
\end{proof}
 	
\begin{proof}[End of the proof of Theorem \ref{thm2}]
Given the above preparations, we can easily carry out the global case on a compact K\"ahler manifold $X$ and finish the proof of Theorem \ref{thm2} by applying an identical arguments in \cite[Section 3.2]{X}. To make this note more complete and readable, let's give a sketch.

Now assume we are given the setting and notations in Theorem \ref{thm2}. It suffices to check the result for elements in $P^{1,1}(X,\mathbb R):=P^{1,1}(X,\mathbb C)\cap H^{2}(X,\mathbb R)$. Arbitrarily take $[\alpha]\in P^{1,1}(X,\mathbb R)$ with a smooth representative $\alpha$, i.e. $[\Omega]\wedge[\eta]\wedge[\alpha]=0$. Then we obviously have
\begin{equation}\label{laplace-condition}
\int_X\Omega\wedge \eta\wedge\alpha=0.
\end{equation}
On the other hand, Lemma \ref{lem_c3}(1) implies that the following is a second order elliptic equation of $\phi$:
\begin{equation}\label{laplace}
\Omega\wedge \eta\wedge\sqrt{-1}\partial\bar\partial\phi=-\Omega\wedge \eta\wedge\alpha.
\end{equation}
Since both $\Omega$ and $\eta$ are closed, the compatibility condition of \eqref{laplace} is exactly \eqref{laplace-condition}. Then by standard elliptic theory one obtains a smooth solution $\phi\in C^\infty(X,\mathbb R)$ to the above equation \eqref{laplace}, i.e.
$$\Omega\wedge \eta\wedge(\alpha+\sqrt{-1}\partial\bar\partial\phi)=0.$$
Then Lemma \ref{lem_c3}(2) implies 
\begin{equation}\label{pt_ineq}
\Omega\wedge(\alpha+\sqrt{-1}\partial\bar\partial\phi)^2\le0
\end{equation}
on $X$, and the equality holds at some $p\in X$ if and only if $\alpha+\sqrt{-1}\partial\bar\partial\phi=0$ at $p$. Integrating \eqref{pt_ineq} gives
$$Q([\alpha],[\alpha])=\int_X\Omega\wedge(\alpha+\sqrt{-1}\partial\bar\partial\phi)^2\le0,$$
and the equality holds if and only if $\alpha+\sqrt{-1}\partial\bar\partial\phi=0$ everywhere on $X$, i.e. $[\alpha]=0$ in $H^{1,1}(X,\mathbb R)$.
\par The proof is completed.
\end{proof}

\begin{rem}\label{rem_ques}
(1) From the above discussions, to obtain general abstract elements in $\mathscr H$, a natural approach may be characterizing the positive $(m,m)$-forms, say $\Phi\in \Lambda^{m,m}_{\mathbb R}(\mathbb C^n)$, with the property that the following homogeneous polynomial defined on $\Lambda^{1,1}_{\mathbb R}(\mathbb C^n)$,
$$P_{\Phi}(x):=\Phi\wedge x^{n-m},$$
is hyperbolic (and complete). Such a $\Phi$, together with $n-m-1$ elements in $C(P_\Phi)$, could produce general abstract elements in $\mathscr H$. 

(2) Relating to the above problem, it is also natural to consider the following equation on a compact K\"ahler manifold $X$:
\begin{equation}\label{mix_hess}
\Phi\wedge(\alpha+\sqrt{-1}\partial\bar\partial\phi)^{n-m}=\Psi,
\end{equation}
where $\Psi$ is a smooth volume from on $X$, and $\alpha$ may be assumed to be $(n-m)$-positive with respect to $(\omega_X,\Phi)$. The case that $\Phi=\omega_X^{m}$ is just the complex Hessian equation, and the case that $\Phi=\omega_1\wedge...\wedge\omega_{m}$ with $\omega_j$'s K\"ahler metrics on $X$ has been proposed in \cite[Remark 2.12]{X}. It seems interesting to explore the relations between the solvability of \eqref{mix_hess} and the hyperbolicity of $P_\Phi$, which may have applications in understanding the structure of $\mathscr H$. 
\end{rem}

\subsection{Khovanskii-Teissier type inequalities}\label{subsect_kt}

\begin{rem}\label{rem_cor}
We remark some consequences of the Hodge index theorem, which should be well-known (see e.g. \cite{DN,X}). 
\begin{itemize}
\item[(1)] For any $([\Omega],[\eta])\in\mathscr{H}$ with $[\Omega\wedge\eta]\neq0$, we have the Hard Lefschetz and Lefschetz Decomposition Theorems (with respect to $([\Omega],[\eta])$) as follows:
\begin{itemize}
\item[(a)] The map $[\Omega]:H^{1,1}(X,\mathbb C)\to H^{n-1,n-1}(X,\mathbb C)$ is an isomorphism;
\item[(b)] The space $H^{1,1}(X,\mathbb C)$ has a $Q$-orthogonal direct sum decomposition
$$H^{1,1}(X,\mathbb C)=P^{1,1}(X,\mathbb C)\oplus\mathbb C[\eta].$$
\end{itemize}
For (a) it suffices to check that it is injective. Assume $[\alpha]\in H^{1,1}(X,\mathbb C)$ with $[\Omega]\wedge[\alpha]=0$, then obviously there hold $[\alpha]\in P^{1,1}(X,\mathbb C)$ and $Q([\alpha],[\alpha])=0$, and hence by Hodge index theorem $[\alpha]=0$. This proves (a).\\
For (b) we consider the map $[\Omega]\wedge[\eta]:H^{1,1}(X,\mathbb C)\to H^{n,n}(X,\mathbb C)$, whose kernel is exactly $P^{1,1}(X,\mathbb C)$. Moreover, restricting to the subspace $\mathbb C[\eta]$ this map specifies an isomorphism between $\mathbb C[\eta]$ and $H^{n,n}(X,\mathbb C)$, as $H^{n,n}(X,\mathbb C)\cong\mathbb C$ and we can check that $[\Omega]\wedge[\eta]\wedge[\eta]\neq0$. Therefore, $H^{1,1}(X,\mathbb C)=P^{1,1}(X,\mathbb C)\oplus\mathbb C[\eta]$, which is obviously $Q$-orthogonal. This proves (b).
\item[(2)] For any $([\Omega],[\eta])\in\mathscr{H}$ with $[\Omega\wedge\eta]\neq0$, we have the Khovanskii-Teissier type inequalities as follows. 
\begin{itemize}
\item[(c)] For any closed real $(1,1)$-forms $\phi,\psi\in H^{1,1}(X,\mathbb R)$ with $\phi$ $2$-positive with respect to $(\Omega,\eta)$, then 
$$\left(\int_X\Omega\wedge\phi\wedge\psi\right)^2\ge\left(\int_X\Omega\wedge\phi^2\right)\left(\int_X\Omega\wedge\psi^2\right)$$
with equality if and only if $[\phi]$ and $[\psi]$ are proportional.
\end{itemize}
Indeed, using Hodge index theorem for $([\Omega],[\eta])$, one can easily check that the polynomial $P([\phi]):=\int_{X}\Omega\wedge\phi^2$ defined on $H^{1,1}(X,\mathbb R)$ is hyperbolic at $[\eta]$ and complete, with $P([\eta])>0$. Then if $\phi$ is $2$-positive with respect to $(\Omega,\eta)$, we have $[\phi]\in C(P,[\eta])$. Now the desired conclusion follows from G{\aa}rding inequality and the similar arguments in Lemma \ref{lem_c3}(2).
\end{itemize}
\end{rem}

Combining Theorem \ref{thm2} and Remark \ref{rem_cor}(2), we immediately conclude the following Khovanskii-Teissier type inequality.
\begin{thm}\label{(1,1)-kt}
Assume $\lambda_1,...,\lambda_N,\lambda_{N+1}\in\mathbb Z_{\ge1}$ with $\sum_{i=1}^{N+1}\lambda_s= n-2$. Assume $\alpha_{ij_i}, 1\le i\le N+1$ and $1\le j_i\le\lambda_i+1$, be closed real $(1,1)$-forms on $X$ such that $\alpha_{1j_1}$'s are K\"ahler metrics, and when $2\le i\le N+1$, $\alpha_{ij_i}$ is $(n-\sum_{s=1}^{i-1}\lambda_s)$-positive with respect to\\ $(\bigwedge_{1\le s\le i-1,1\le j_s\le \lambda_s}\alpha_{sj_s},\alpha_{i-1,\lambda_{i-1}+1})$. Denote $\Omega:=\bigwedge_{1\le i\le N+1,1\le j_i\le \lambda_i}\alpha_{ij_i}$. Assume $\eta$ be a closed real $(1,1)$-forms on $X$ which is $2$-positive with respect to  $(\Omega,\alpha_{N+1,\lambda_{N+1}+1})$. Then for any $[\alpha]\in H^{1,1}(X,\mathbb R)$, we have
$$\left(\int_X\Omega\wedge\eta\wedge\alpha\right)^2\ge\left(\int_X\Omega\wedge\eta^2\right)\left(\int_X\Omega\wedge\alpha^2\right),$$
and the equality holds if and only if $[\eta]$ and $[\alpha]$ are proportional.
\end{thm}

\begin{exmp}
Let's look at the special case corresponding to Theorem \ref{thm1}. Fix an integer $m\le n-2$. Assume $\omega_1,...,\omega_{m}$ are K\"ahler metrics and $\alpha_1,...,\alpha_{n-m-1}$ closed real $(1,1)$-forms on $X$ such that every $\alpha_j$ is $(n-m)$-positive with respect to $(\omega_1\wedge...\wedge\omega_{m},\omega_X)$. Set $\Omega:=\omega_1\wedge...\wedge\omega_{m}\wedge\alpha_1\wedge...\wedge\alpha_{n-m-2}$. Then for any $[\alpha]\in H^{1,1}(X,\mathbb R)$,
$$\left(\int_X\Omega\wedge\alpha_{n-m-1}\wedge\alpha\right)^2\ge\left(\int_X\Omega\wedge\alpha_{n-m-1}^2\right)\left(\int_X\Omega\wedge\alpha^2\right),$$
and the equality holds if and only if $[\alpha_{n-m-1}]$ and $[\alpha]$ are proportional.
\end{exmp}
Similarly, we have the following log-convavity, which generalizes \cite[Theorem B]{X} and \cite[Corollary 2.16]{Co} to a more mixed/polarized and degenerate setting.
\begin{thm}
Assume $\lambda_1,...,\lambda_N,\lambda_{N+1}\in\mathbb Z_{\ge1}$ with $\sum_{i=1}^{N+1}\lambda_s= n-2$. Assume $\alpha_{ij_i}, 1\le i\le N$ and $1\le j_i\le\lambda_i+1$, be closed real $(1,1)$-forms on $X$ such that $\alpha_{1j_1}$'s are K\"ahler metrics, and when $2\le i\le N$, $\alpha_{ij_i}$ is $(n-\sum_{s=1}^{i-1}\lambda_s)$-positive with respect to\\ $(\bigwedge_{1\le s\le i-1,1\le j_s\le \lambda_s}\alpha_{sj_s},\alpha_{i-1,\lambda_{i-1}+1})$. Denote $\tilde\Omega:=\bigwedge_{1\le i\le N,1\le j_i\le \lambda_i}\alpha_{ij_i}$. Assume $\alpha,\beta$ be two closed real $(1,1)$-forms on $X$ which is $(\lambda_{N+1}+2)$-positive with respect to $(\tilde\Omega,\alpha_{N,\lambda_{N}+1})$ (note that $\lambda_{N+1}+2=n-\sum_{s=1}^{N}\lambda_s$), and set $a_k:=\int_X\tilde\Omega\wedge\alpha^k\wedge\beta^{\lambda_{N+1}+2-k},0\le k\le \lambda_{N+1}+2$. Then for any $1\le k\le \lambda_{N+1}+1$ we have
$$a_k^2\ge a_{k-1}a_{k+1}$$
with equality holding for some $k$ if and only if $[\alpha]$ and $[\beta]$ are proportional.
\end{thm}
\begin{proof}
Note that if $\alpha,\beta$ are $(\lambda_{N+1}+2)$-positive with respect to $(\tilde\Omega,\alpha_{N,\lambda_{N}+1})$, then by Lemma \ref{thm2.4}(1), both $\alpha,\beta$ are $2$-positive with respect to $(\tilde\Omega\wedge\alpha^{k-1}\wedge\beta^{\lambda_{N+1}+1-k},\alpha)$ for each  $1\le k\le \lambda_{N+1}+1$. Therefore, the result follows from Theorem \ref{(1,1)-kt} immediately.
\end{proof}

\begin{exmp}
Let's consider $N=1$ case in Theorem 2.11, and denote $\lambda_1=m$ (then $\lambda_2=n-m-2$). Write $\alpha_{1j}=:\omega_j, j=1,...,m$, which are Kahler metrics; and assume $\alpha,\beta$ are closed real $(1,1)$-forms which are $(n-m)$-positive with respect to $(\omega_1\wedge...\wedge\omega_m,\omega_X)$. Then Theorem 2.11 reads, after setting $a_k:=\int_X\omega_1\wedge...\wedge\omega_m\wedge\alpha^k\wedge\beta^{n-m-k}$, 
$$a_k^2\ge a_{k-1}a_{k+1} \,\,for\,\,each\,\,k=1,2,...,n-m-1.$$
These log-concavity results generalize the ones in \cite[Theorem B]{X} and \cite[Corollary 2.16]{Co}, in which the case when $\omega_1=...=\omega_m=\omega_X$ and $\alpha,\beta$ are $(n-m)$-positive with respect to $\omega_X$ is settled.
\end{exmp}

\section{Hodeg-Riemann bilinear relation: Proof of Theorem \ref{thm3}}\label{sect_hr}

\subsection{A linear version of Theorem \ref{thm3}}\label{linearversion}
A proof for Theorem \ref{thm3} can be achieved by adapting the arguments \cite{T,DN,DN13,X2}. The first step is to prove a linear/local version of Theorem \ref{thm3} on $\mathbb C^n$ (see Proposition \ref{prop-local} below). To work with $\mathbb C^n$, similar to the previous sections, we denote $\Lambda^{p,q}(\mathbb C^n)$ be the space of $(p,q)$-forms of constant coefficients on $\mathbb C^n$ and $\Lambda^{p,p}_{\mathbb R}(\mathbb C^n)$ be the space of real $(p,p)$-forms of constant coefficients on $\mathbb C^n$; for a given pair $(\Omega,\eta)\in\Lambda^{n-k,n-k}_{\mathbb R}(\mathbb C^n)\times\Lambda^{1,1}_{\mathbb R}(\mathbb C^n)$ and $(p,q)$ with $0\le p,q\le p+q\le n$ and $p+q=k$, we define the primitive space with respect to $(\Omega,\eta)$ in $\Lambda^{p,q}(\mathbb C^n)$ by
$$P^{p,q}=P^{p,q}_{(\Omega,\eta)}:=\{\gamma\in\Lambda^{p,q}(\mathbb C^n)|\Omega\wedge\eta\wedge\gamma=0\,\,on\,\,\mathbb C^n\},$$
and we say that $(\Omega,\eta)$ satisfies the Hodge-Riemann bilinear relation on $\Lambda^{p,q}(\mathbb C^n)$ if the quadratic form
$$Q(\beta,\eta):=\sqrt{-1}^{q-p}(-1)^{\frac{(p+q)(p+q+1)}{2}}\Omega\wedge\beta\wedge\bar\eta$$
is positive definite on $P^{p,q}_{(\Omega,\eta)}$. Let $\mathscr H^{n-k,n-k}$ be the set of pair $(\Omega,\eta)\in\Lambda^{n-k,n-k}_{\mathbb R}(\mathbb C^n)\times\Lambda^{1,1}_{\mathbb R}(\mathbb C^n)$ satifying the Hodge-Riemann bilinear relation on $\Lambda^{p,q}(\mathbb C^n)$ for any $(p,q)$ with $0\le p,q\le p+q\le n$ and $p+q=k$.

The following is a linear version of Theorem \ref{thm3} that we need later.
\begin{prop}\label{prop-local}
Given $N\in \mathbb Z_{\ge0}$. Assume $\lambda_0=0$ and $\lambda_1,...,\lambda_N\in\mathbb Z_{\ge1}$ with $\sum_{i=0}^N\lambda_i=n-(p+q)$, $p,q\in\mathbb Z_{\ge1}$. For $1\le i\le N$, assume $\alpha_{ij_i}\in\Lambda^{1,1}_{\mathbb R}(\mathbb C^n)$, $1\le j_i\le\lambda_i$, be semi-positive real $(1,1)$-form on $\mathbb C^n$ such that $\alpha_{ij_i}$ has at least $(n-\sum_{s=0}^{i-1}\lambda_s)$ positive eigenvalues. Denote $\Omega:=\bigwedge_{1\le i\le N,1\le j_i\le \lambda_i}\alpha_{ij_i}$ if $N\ge1$ and $\Omega=1\in\mathbb R$ if $N=0$. Assume $\eta_1,...,\eta_{p+q-2},\eta_{p+q-1}\in\Lambda^{1,1}_{\mathbb R}(\mathbb C^n)$ are semi-positive real $(1,1)$-forms on $\mathbb C^n$ of at least $(n-\sum_{s=0}^{N}\lambda_s)$ positive eigenvalues (note that $p+q-2=n-2-\sum_{i=1}^N\lambda_s$). Then for $0\le d\le[(p+q)/2]-1$,
$$(\Omega\wedge\eta_1\wedge...\wedge\eta_{2d},\eta_{2d+1})\in\mathscr H^{n-p-q+2d,n-p-q+2d}.$$
\end{prop}

In fact, the $N=0$ case is exactly Timorin's work \cite{T}, while $N=1$ case is (implicitly) Xiao's \cite[Theorem 3.1,Remarks 2.2 and 3.9]{X2}. 

Basing on the $N=0$ case, we can take an induction argument similar to \cite{T} to prove the general case. We particularly remark that we will do induction with respect to the number $N$, while the original reduction argument in \cite{T} (and the modified one in \cite{X2}) was made with respect to the dimension of manifolds. It seems that our argument gives a \emph{slightly} different proof for \cite[Theorem 3.1]{X2}.

First of all, let's collect two useful linear algebra lemmas from \cite{X2} as follows.

For any $v\in\mathbb C^n\setminus \{0\}$, $H_v$ is the hyperplane defined by 
\begin{align}\label{H_v}
H_v:=\{x\in\mathbb C^n|v\cdot x=0\}.
\end{align}

\begin{lem}\cite[Lemma 2.1]{X2}\label{lem_la1}
For any positive integral $r<n$ and $\beta\in\Lambda^{1,1}_{\mathbb R}(\mathbb C^n)$, which is semi-positive and has at least $r$ positive eigenvalues, there is a proper subspace $S(\beta)$ of $\mathbb C^n$ such that for any $v\in\mathbb C^n\setminus S(\beta)$, $\beta_{|H_v}$, the restriction of $\beta$ to $H_v$, is semi-positive and has at least $r$ positive eigenvalues.
\end{lem}
\begin{proof}
This can be checked by the argument in \cite[Lemma 2.1]{X2} by using hyperbolicity of the homogeneous polynomial $P_{m}$ defined in subsection \ref{hi_pf}. Alternatively, one may check this lemma directly as $\beta$ is semi-positive.
\end{proof}

\begin{lem}\cite[Lemma 2.1]{X2}\label{lem_la2}
Given any finitely many hyperplanes $H_{v_1},...,H_{v_d}$ in $\mathbb C^n$, there exists an orthonormal basis $\{e_1,...,e_n\}$ of $\mathbb C^n$ such that
$$e_i\in\mathbb C^n\setminus\cup_{\lambda=1}^d H_{v_\lambda}$$
\end{lem}

Next step is to show that Proposition \ref{prop-local} for $N=k-1$ implies Hard Lefschetz and Lefschetz Decomposition Theorems for $N=k$.

\begin{prop}\label{prop-local'}(compare \cite[Proposition 1]{T} and \cite[Lemma 3.3]{X2})
Given $k\ge1$. Assume Proposition \ref{prop-local} hold for $N=k-1$ in any dimensions. Given the same data and conditions in Proposition \ref{prop-local} for $N=k$, then the corresponding Hard Lefschetz Theorem holds for $\Omega\wedge\eta_1\wedge...\wedge\eta_{2d}$, i.e.
$$\Omega\wedge\eta_1\wedge...\wedge\eta_{2d}:\Lambda^{p-d,q-d}(\mathbb C^n)\to\Lambda^{n-q+d,n-p+d}(\mathbb C^n)$$
is an isomorphism.
\end{prop}
\begin{proof}
It suffices to check the injectivity due to the dimension reason. Let $\Phi\in\Lambda^{p-d,q-d}(\mathbb C^n)$ satisfying
\begin{align}\label{eq_inj}
\Omega\wedge\eta_1\wedge...\wedge\eta_{2d}\wedge\Phi=\alpha_{11}\wedge...\wedge\alpha_{1\lambda_1}\wedge\alpha_{21}\wedge...\wedge\alpha_{k\lambda_k}\wedge\eta_1...\wedge\eta_{2d}\wedge\Phi=0.
\end{align}
The goal is to show that $\Phi=0$. Loosely speaking, the idea to achieve this goal is to ``increase" the positivity of the involved $(1,1)$-forms by decreasing the dimension of the ambient space and hence Proposition \ref{prop-local} for $N=k-1$ case could be applied. The followings are the details.

Firstly, since $\lambda_1\ge1$, we apply Lemma \ref{lem_la1} to see that for a generic $(n-1)$-dimensional subspace $H^{n-1}=H^{n-1}_{v}$ of $\mathbb C^n$ (here being generic means $v$ can be arbitrarily chosen outside finitely many proper subspaces in $\mathbb C^n$), for $2\le i\le k$, $\alpha_{i\cdot|H}$ is semi-positive and has at least $(n-\sum_{s=1}^{i-1}\lambda_s)$ positive eigenvalues, here $\cdot$ ranges from $1$ to $\lambda_i$, and $\eta_{1|H},...,\eta_{p+q-1|H}$ is semi-positive and has at least $(n-\sum_{s=1}^{k}\lambda_s)$ positive eigenvalues. Of course $\alpha_{1\cdot|H}$ keeps positive on $H$. Also we have, on $H$,
\begin{align}
\alpha_{11|H}\wedge...\wedge\alpha_{1\lambda_1|H}\wedge\alpha_{21|H}\wedge...\wedge\alpha_{k\lambda_k|H}\wedge\eta_{1|H}...\wedge\eta_{2d|H}\wedge\Phi_{|H}=0.
\end{align}
Repeating the above procedure $\lambda_1$ times, we see that for a generic $(n-\lambda_1)$-dimensional subspace $L^{n-\lambda_1}$ of $\mathbb C^n$, for $i=1,2$, $\alpha_{i\cdot|L}$ is positive on $L$, and for $3\le i\le k$, $\alpha_{i\cdot|L}$ is semi-positive and has at least $(n-\sum_{s=1}^{i-1}\lambda_s)$ positive eigenvalues, and $\eta_{1|L},...,\eta_{p+q-1|L}$ is semi-positive and has at least $(n-\sum_{s=1}^{k}\lambda_s)$ positive eigenvalues. Also we have, on $L$,
\begin{align}
\alpha_{11|L}\wedge...\wedge\alpha_{1\lambda_1|L}\wedge\alpha_{21|L}\wedge...\wedge\alpha_{k\lambda_k|L}\wedge\eta_{1|L}...\wedge\eta_{2d|L}\wedge\Phi_{|L}=0.
\end{align}
Set $\Theta_{|L}:=\widehat{\alpha_{11|L}}\wedge...\wedge\alpha_{1\lambda_1|L}\wedge\alpha_{21|L}\wedge...\wedge\alpha_{k\lambda_k|L}\wedge\eta_{1|L}...\wedge\eta_{2d|L}$, i,e, removing $\alpha_{11|L}$, then $\Phi_{|L}$ is primitive with respect to $(\Theta_{|L},\alpha_{11|L})$. Since both $\alpha_{1\cdot|L}, \alpha_{2\cdot|L}$ are positive on $L$, and for $i\ge3$, 
$$n-\sum_{s=1}^{i-1}\lambda_s=(n-\lambda_1)-((\lambda_1-1+\lambda_2)+...+\lambda_{i-1})+(\lambda_1-1)\ge(n-\lambda_1)-((\lambda_1-1+\lambda_2)+...+\lambda_{i-1}),$$
we can apply Proposition \ref{prop-local} for $N=k-1$ case in dimension $n-\lambda_1$ to conclude that
\begin{align}\label{eq_inj3}
c_d\cdot\Theta_{|L}\wedge\Phi_{|L}\wedge\overline\Phi_{|L}\ge0.
\end{align}
where $c_d:=\sqrt{-1}^{q-p}(-1)^{\frac{(p+q-2d)(p+q-2d+1)}{2}}$. This equivalently means that for a generic $(n-\lambda_1+1)$-dimensional subspace $K$ of $\mathbb C^n$ and a generic $(n-\lambda_1)$-dimensional subspace $L$ of $K$,
\begin{align}
c_d\cdot\Theta_{|K}\wedge\Phi_{|K}\wedge\overline\Phi_{|K}\wedge\sqrt{-1}dL\wedge d\bar L\ge0.
\end{align}
Now by Lemma \ref{lem_la2} we fix an orthonormal basis $f_1,...,f_{n-\lambda_1+1}$ of $K^{n-\lambda_1+1}$ such that for every $1\le l\le n-\lambda_1+1$,
\begin{align}\label{eq_f_l}
c_d\cdot\Theta_{|K}\wedge\Phi_{|K}\wedge\overline\Phi_{|K}\wedge\sqrt{-1}dL_{f_l}\wedge d\bar L_{f_l}\ge0,
\end{align}
where $L_{f_l}$ is the hyperplane in $K$ defined by $f_l\in K$ (see \eqref{H_v}). May assume $\alpha_{11|K}=\sum_{l=1}^{n-\lambda_1+1}\sqrt{-1}dL_{f_l}\wedge d\bar L_{f_l}$, then summing up \eqref{eq_f_l} over $l$ gives

\begin{align}\label{eq_inj6}
c_d\cdot\alpha_{11|K}\wedge...\wedge\alpha_{1\lambda_1|K}\wedge\alpha_{21|K}\wedge...\wedge\alpha_{k\lambda_k|K}\wedge\eta_{1|K}...\wedge\eta_{2d|K}\wedge\Phi_{|K}\wedge\overline{\Phi}_{|K}\ge0.
\end{align}
However, by assumption \eqref{eq_inj} the equality holds in \eqref{eq_inj6}, which forces \eqref{eq_inj3} with $L=L_{f_l}$ to be an equality. Therefore, by Proposition \ref{prop-local} for $N=k-1$ case we obtain $\Phi_{|L_{f_l}}=0$, which, since $p+q-2d\le p+q=n-\sum_{i=1}^k\lambda_i\le n-\lambda_1<n-\lambda_1+1$, implies that $\Phi_{|K}=0$, and eventually we conclude that $\Phi=0$ on $\mathbb C^n$.

Proposition \ref{prop-local'} is proved.
\end{proof}

Given the Hard Lefschetz Theorem for $N=k$ case, we immediately have the Lefschetz Decomposition Theorem for $N=k$ by the same arguments in \cite[Corollaries 2 and 3]{T} (also see \cite[Lemma 3.6]{X2}).
\begin{prop}\label{prop-local''}
Assume the same data and conditions in Proposition \ref{prop-local} for $N=k$, and let the quadratic form $Q$ and primitive space $P^{p-d,q-d}$ are defined with respect to $(\Omega\wedge\eta_1\wedge...\wedge\eta_{2d},\eta_{2d+1})$. Then the corresponding Lefschetz Decomposition Theorem holds, i.e.
\begin{itemize}
\item[(1)] $\Lambda^{p-d,q-d}(\mathbb C^n)$ has a $Q$-orthogonal direct sum decomposition
$$\Lambda^{p-d,q-d}(\mathbb C^n)=P^{p-d,q-d}(\mathbb C^n)\oplus\eta_{2d+1}\wedge\Lambda^{p-d-1,q-d-1}(\mathbb C^n).$$
\item[(2)] $dimP^{p-d,q-d}=dim\Lambda^{p-d,q-d}-dim\Lambda^{p-d-1,q-d-1}$.
\end{itemize}
\end{prop}

Given the above preparations, we now present  a proof for Proposition \ref{prop-local} by adapting the homotopy argument in \cite[Section 6]{T}.
\begin{proof}[Proof of Proposition \ref{prop-local}]
The $N=0$ case is exactly \cite{T}. Assume Proposition \ref{prop-local} holds for $N=k-1$ in any dimensions. Given the data and conditions in $N=k$ case, we consider the deformation
$$\Theta_t:=\left\{\begin{array}{cc} \alpha_{11}\wedge...\wedge\alpha_{1\lambda_1}\wedge\eta^t_1...\wedge\eta^t_{2d}, & k=1;\\
\alpha_{11}\wedge...\wedge\alpha_{1\lambda_1}\wedge\alpha^t_{21}\wedge...\wedge\alpha^t_{2\lambda_2}\wedge\alpha_{31}\wedge...\wedge\alpha_{k\lambda_k}\wedge\eta_1...\wedge\eta_{2d}, & k\ge2,\end{array}\right.$$
where $\eta^t_{j}:=(1-t)\eta_j+t\alpha_{11}$, $1\le j\le2d$, and $\alpha^t_{2j_2}:=(1-t)\alpha_{2j_2}+t\alpha_{11}$, $1\le j_2\le\lambda_2$, and $t\in[0,1]$. Let $P_t^{p-d,q-d}$ be the primitive space with respect to $(\Theta_t,\eta_{2d+1})$ and $Q_t$ the quadratic form defined by $\Theta_t$. By Proposition \ref{prop-local'} we know 
$Q_t$ is non-degenerate for $t\in[0,1]$, while by Proposition \ref{prop-local''}, $dim P_t^{p-d,q-d}$ keeps the same for $t\in[0,1]$. Moreover, applying the result for $N=k-1$ case gives that $Q_1$ is positive definite on $P_1^{p-d,q-d}$ (note that $\eta^1_j=\alpha_{11}$ is a positive real $(1,1)$-form for every $j=1,...,2d$, and $\alpha^1_{2j_2}=\alpha_{11}$ is also a positive real $(1,1)$-form for every $1\le j_2\le\lambda_2$), therefore, we conclude that $Q_0$ is positive definite on $P_0^{p-d,q-d}$, which is exactly the required result.

Proposition \ref{prop-local} is proved. 
\end{proof}

Using Propositions \ref{prop-local}, \ref{prop-local'} and \ref{prop-local''} and the same arguments in \cite[Proposition 2.2]{DN}, we also have the following.

\begin{lem}\label{lem-l2}
Given the same data and conditions in Proposition \ref{prop-local}, there exists a positive number $C\ge1$ such that
$$\|\Omega\wedge\eta_1\wedge...\wedge\eta_{2d}\wedge\eta_{2d+1}\wedge\Phi\|^2+Q(\Phi,\Phi)\ge C^{-1}\|\Phi\|^2,\,\,\,\,\,\Phi\in\Lambda^{p-d,q-d}(\mathbb C^n).$$
\end{lem}

\subsection{Proof of Theorem \ref{thm3}}
Given the results in the above subsection \ref{linearversion} (including Proposition \ref{prop-local} and Lemma \ref{lem-l2}), Theorem \ref{thm3} can be proved by the similar arguments in \cite[Propositions 2.3 and 2.4, Section 3]{DN} (also see \cite[Section 3.2]{X2}). For the sake of completeness we present a proof here. 

\begin{proof}[Proof of Theorem \ref{thm3}]
Given $[\Phi]\in P^{p-d,q-d}$ with $\Phi$ a smooth representative, by Lemma \ref{lem-l2} and the same $L^2$-method arguments in \cite[Propositions 2.3 and 2.4]{DN} (Lemma \ref{lem-l2} provides, after an integration over $X$, the key $L^2$-estimate which makes the $L^2$-method applicable in solving \eqref{dd^c} below), we can find a $v\in\Lambda^{p-d-1,q-d-1}(X,\mathbb C)$ satisfying
\begin{equation}\label{dd^c}
\Omega\wedge\eta_1\wedge...\wedge\eta_{2d}\wedge\eta_{2d+1}\wedge(\Phi-\sqrt{-1}\partial\bar\partial v)=0\,\,\,\,\,on\,\,\,X.
\end{equation}
In case $p-d=0$ or $q-d=0$ we replace $\sqrt{-1}\partial\bar\partial v$ by $0$. Then by Proposition \ref{prop-local} we see that
\begin{align}\label{nonneg-pt}
c_d\cdot\Omega\wedge\eta_1\wedge...\wedge\eta_{2d}\wedge(\Phi-\sqrt{-1}\partial\bar\partial v)\wedge\overline{(\Phi-\sqrt{-1}\partial\bar\partial v)}\ge0
\end{align}
holds pointwise on X, integrating which and applying Stokes Formula immediately give
$$Q([\Phi],[\Phi])=\int_Xc_d\cdot\Omega\wedge\eta_1\wedge...\wedge\eta_{2d}\wedge(\Phi-\sqrt{-1}\partial\bar\partial v)\wedge\overline{(\Phi-\sqrt{-1}\partial\bar\partial v)}\ge0.$$
Moreover, the equality holds if and only if \eqref{nonneg-pt} is an equality everywhere on $X$, and hence by Proposition \ref{prop-local} if and only if $\Phi=\sqrt{-1}\partial\bar\partial v$ on $X$, i.e. $[\Phi]=0$ in $H^{p-d,q-d}$.

Theorem \ref{thm3} is proved.
\end{proof}

\begin{rem}
We can also slightly generalize the abstract versions of the mixed Hodge-Riemann bilinear relations in \cite[Theorem 1.1]{DN13} and \cite[Theorem 4.3]{X2} to the setting of Theorem \ref{thm3}.
\end{rem}

\section{Higher-rank Khovanskii-Teissier inequality}\label{sect_kt}
Given results in Sections \ref{sect_hi} and \ref{sect_hr}, we can now prove the higher-rank Khovanskii-Teissier inequality in Theorem \ref{kt} by carefully adapting Li's arguments in \cite[Section 2.2]{Li16}. 

\begin{lem}
Assume $\lambda_0=0$ and $\lambda_1,...,\lambda_N\in\mathbb Z_{\ge1}$ with $\sum_{i=1}^N\lambda_s=n-2p$, $p\in\mathbb Z_{\ge1}$. Assume $\alpha_{ij_i}, 1\le i\le N$ and $1\le j_i\le\lambda_i$, be semi-positive closed real $(1,1)$-forms on $X$ such that $\alpha_{ij_i}$ has at least $(n-\sum_{s=0}^{i-1}\lambda_s)$ positive eigenvalues. Denote $\Omega:=\bigwedge_{1\le i\le N,1\le j_i\le \lambda_i}\alpha_{ij_i}$. Assume $\eta$ is a semi-positive closed real $(1,1)$-forms on $X$ of at least $(n-\sum_{s=0}^{N}\lambda_s)$ positive eigenvalues, and $\alpha$ a closed real $(1,1)$-forms on $X$ which is $2$-positive with respect to $(\Omega\wedge\eta^{2p-2},\alpha_{N,\lambda_N})$. Assume $p\ge2$. Then for any given $[\gamma]\in H^{p,p}(X,\mathbb C)$, there holds
\begin{align}\label{decom0}
[\gamma]=[\gamma_p]+[\eta]\wedge[\gamma_{p-1}]+[\eta^2]\wedge[\gamma_{p-2}]+...+[\eta^{p-1}]\wedge[\gamma_{1}]+\mu\cdot[\eta^{p-1}]\wedge[\alpha],
\end{align}
where $[\gamma_s]\in P^{s,s}_{([\Omega\wedge\eta^{2(p-s)}],[\eta])}(X,\mathbb C)$ for $2\le s\le p$, $[\gamma_1]\in P^{1,1}_{([\Omega\wedge\eta^{2(p-1)}],[\alpha])}(X,\mathbb C)$, and $\mu\in\mathbb C$.
\end{lem}
\begin{proof}
In Theorem \ref{thm3} and Remark \ref{rem_hr}, we choose $\eta_1=...=\eta_{2p-2}=\eta$, then Remark \ref{rem_hr}(b) immediately implies
\begin{align}\label{decom1}
[\gamma]&=[\gamma_p]+[\eta]\wedge[\tilde\gamma_{p-1}]\nonumber\\
&=[\gamma_p]+[\eta]\wedge([\gamma_{p-1}]+[\eta]\wedge[\tilde\gamma_{p-2}])\nonumber\\
&=[\gamma_p]+[\eta]\wedge[\gamma_{p-1}]+[\eta^2]\wedge[\tilde\gamma_{p-2}]\nonumber\\
&=...\nonumber\\
&=[\gamma_p]+[\eta]\wedge[\gamma_{p-1}]+[\eta^2]\wedge[\gamma_{p-2}]+...+[\eta^{p-2}]\wedge[\gamma_{2}]+[\eta^{p-1}]\wedge[\tilde\gamma_{1}].
\end{align}
with $[\gamma_s]\in P^{p,p}_{([\Omega\wedge\eta^{2(p-s)}],[\eta])}(X,\mathbb C)$, $2\le s\le p$ and $[\tilde\gamma_{1}]\in H^{1,1}(X,\mathbb C)$. 

To proceed, we apply Theorem \ref{thm2} to see that $([\Omega\wedge\eta^{2(p-1)}],[\alpha])$ satisfies the Hodge index theorem, and hence by Remark \ref{rem_cor}(b), we can decompose $[\tilde\gamma_1]$ as
\begin{align}\label{decom2}
[\tilde\gamma_1]=[\gamma_1]+\mu[\alpha]
\end{align}
with $[\gamma_1]\in P^{1,1}_{([\Omega\wedge\eta^{2(p-1)}],[\alpha])}(X,\mathbb C)$ and $\mu\in\mathbb C$.

Plugging \eqref{decom2} into \eqref{decom1} gives the required decomposition \eqref{decom0}.
\end{proof}

Then we may express the integrands involved in \eqref{kt} as follows.

\begin{lem}\label{lem_eq}
The followings hold.
\begin{itemize}
\item[(1)] $[\Omega]\wedge[\eta^{p-1}]\wedge[\alpha]\wedge[\gamma]=\mu\cdot[\Omega]\wedge[\eta^{2(p-1)}]\wedge[\alpha^2]$.
\item[(2)] $[\Omega]\wedge[\gamma]\wedge[\bar\gamma]=|\mu|^2\cdot[\Omega]\wedge[\eta^{2(p-1)}]\wedge[\alpha^2]+\sum_{s=1}^p[\Omega]\wedge[\eta^{2(p-s)}]\wedge[\gamma_s]\wedge[\bar\gamma_{s}]$.
\end{itemize}
\end{lem}

\begin{proof}
These are consequences of the decomposition of $[\gamma]$ in \eqref{decom0}.

For (1), we compute
\begin{align}
&[\Omega]\wedge[\eta^{p-1}]\wedge[\alpha]\wedge[\gamma]\nonumber\\
&=[\Omega]\wedge[\eta^{p-1}]\wedge[\alpha]\wedge\left([\gamma_p]+[\eta]\wedge[\gamma_{p-1}]+...+[\eta^{p-1}]\wedge[\gamma_{1}]+\mu[\eta^{p-1}]\wedge[\alpha]\right)\nonumber\\
&=\mu[\Omega]\wedge[\eta^{2(p-1)}]\wedge[\alpha^2]+\sum_{s=1}^p[\Omega]\wedge[\eta^{2p-s-1}]\wedge[\alpha]\wedge[\gamma_s]\nonumber\\
&=\mu[\Omega]\wedge[\eta^{2(p-1)}]\wedge[\alpha^2]+[\Omega]\wedge[\eta^{2(p-1)}]\wedge[\alpha]\wedge[\gamma_1]+\sum_{s=2}^p[\Omega]\wedge[\eta^{2(p-s)+1}]\wedge[\gamma_s]\wedge[\alpha]\wedge[\eta^{s-2}]\nonumber\\
&=\mu[\Omega]\wedge[\eta^{2(p-1)}]\wedge[\alpha^2]\nonumber,
\end{align}
where we have used $[\Omega]\wedge[\eta^{2(p-1)}]\wedge[\alpha]\wedge[\gamma_1]=0$ and $[\Omega]\wedge[\eta^{2(p-s)+1}]\wedge[\gamma_s]=0$ for any $2\le s\le p$.

The item (2) can be checked similarly.
\end{proof}

As an application, we have
\begin{lem}
\begin{align}
&\left(\int_X\Omega\wedge(\eta^{p-1}\wedge\alpha)^2\right)\left(\int_X\Omega\wedge\gamma\wedge\bar\gamma\right)-\left(\int_X\Omega\wedge(\eta^{p-1}\wedge\alpha)\wedge\gamma\right)\left(\int_X\Omega\wedge(\eta^{p-1}\wedge\alpha)\wedge\bar\gamma\right)\nonumber\\
&=\left(\int_X\Omega\wedge(\eta^{p-1}\wedge\alpha)^2\right)\left(\sum_{s=1}^p\int_X[\Omega]\wedge[\eta^{2(p-s)}]\wedge[\gamma_s]\wedge[\bar\gamma_{s}]\right)\nonumber.
\end{align}
\end{lem}
\begin{proof}
Applying Lemma \ref{lem_eq}(1) and (2), the result follows immediately.
\end{proof}

Now we are ready to give an 
\begin{proof}[End of the proof of Theorem \ref{thm_kt}]
Since $[\gamma_s]\in P^{s,s}_{([\Omega\wedge\eta^{2(p-s)}],[\eta])}(X,\mathbb C)$ for $2\le s\le p$, and $[\gamma_1]\in P^{1,1}_{([\Omega\wedge\eta^{2(p-1)}],[\alpha])}(X,\mathbb C)$, by Theorems \ref{thm2} and \ref{thm3} we have, for every for $1\le s\le p$,
\begin{align}
(-1)^s\int_X[\Omega]\wedge[\eta^{2(p-s)}]\wedge[\gamma_s]\wedge[\bar\gamma_{s}]\ge0
\end{align}
with equality holds if and only if $[\gamma_s]=0$. On the other hand, by assumption we also have $\Omega\wedge(\eta^{p-1}\wedge\alpha)^2>0$ on $X$, implying
\begin{align}
\int_X\Omega\wedge(\eta^{p-1}\wedge\alpha)^2>0.
\end{align}

(ii)$\Rightarrow$(i): If the condition (ii) holds, then by Remark \ref{rem_hr} (c), $[\gamma_{2d}]=0$ for $1\le d\le [p/2]$, and hence
\begin{align}
\sum_{s=1}^p\int_X[\Omega]\wedge[\eta^{2(p-s)}]\wedge[\gamma_s]\wedge[\bar\gamma_{s}]&=\sum_{1\le d\le[(p+1)/2]}\int_X[\Omega]\wedge[\eta^{2(p-(2d-1))}]\wedge[\gamma_{2d-1}]\wedge[\bar\gamma_{2d-1}]\nonumber\\
&\le0\nonumber,
\end{align}
with equality holds if and only if $[\gamma_{2d-1}]=0$ for every $1\le d\le[(p+1)/2]$. 

Therefore, the inequality \eqref{kt} is proved for every $[\gamma]\in H^{p,p}(X,\mathbb C)$, with equality holds for some $[\gamma]$ if and only if $[\gamma]$ is proportional to $[\eta^{p-1}\wedge\alpha]$.\\

(i)$\Rightarrow$(ii): Assume a contradiction that there is a $1\le d'\le[p/2]$ with $h^{2d'-1,2d'-1}<h^{2d',2d'}$, then we choose $[\beta]\in P^{2d',2d'}_{([\Omega\wedge\eta^{2(p-2d')}],[\eta])}(X,\mathbb C)\setminus\{0\}$ and consider 
$$[\gamma']:=[\eta^{p-1}\wedge\alpha]+[\eta^{p-2d'}]\wedge[\beta].$$
By the above discussions, it is easy to see that
\begin{align}
&\left(\int_X\Omega\wedge(\eta^{p-1}\wedge\alpha)^2\right)\left(\int_X\Omega\wedge\gamma'\wedge\overline{\gamma'}\right)-\left(\int_X\Omega\wedge(\eta^{p-1}\wedge\alpha)\wedge\gamma'\right)\left(\int_X\Omega\wedge(\eta^{p-1}\wedge\alpha)\wedge\overline{\gamma'}\right)\nonumber\\
&=\left(\int_X\Omega\wedge(\eta^{p-1}\wedge\alpha)^2\right)\left(\int_X[\Omega]\wedge[\eta^{2(p-2d')}]\wedge[\beta]\wedge[\overline{\beta}]\right)\nonumber\\
&>0\nonumber,
\end{align}
which contradicts to condition (i).

Theorem \ref{thm_kt} is proved.
\end{proof}

\begin{rem}\label{general rem}
One can extract from the above proof a general result that the higher-rank Khovanskii-Teissier inequality actually holds with respect to general cohomology classes satisfying certain Hodge-Riemann bilinear relations (we thank an anonymous referee for pointing out this). Precisely, given $[\Omega]\in H^{n-2p,n-2p}(X,\mathbb R)$, $p\ge2$, and $[\eta],[\alpha]\in H^{1,1}(X,\mathbb R)$, assume that $([\Omega\wedge\eta^{2p-2s}],[\eta])$ satisfies the Hodge-Riemann bilinear relation for each $2\le s\le p$, and $([\Omega\wedge\eta^{2p-2}],[\alpha])$ satisfies the Hodge index theorem, then the followings are equivalent:
\begin{itemize}
\item[(i)] For every $[\gamma]\in H^{p,p}(X,\mathbb C)$, there holds
\begin{align}\label{kt=}
\left(\int_X\Omega\wedge(\eta^{p-1}\wedge\alpha)\wedge\gamma\right)\left(\int_X\Omega\wedge(\eta^{p-1}\wedge\alpha)\wedge\bar\gamma\right)\ge\left(\int_X\Omega\wedge(\eta^{p-1}\wedge\alpha)^2\right)\left(\int_X\Omega\wedge\gamma\wedge\bar\gamma\right);
\end{align}
\item[(ii)] For all $1\le l\le[p/2]$, $h^{2l-1,2l-1}=h^{2l,2l}$.
\end{itemize} 
Moreover, if condition (ii) holds, then a $[\gamma]\in H^{p,p}(X,\mathbb C)$ satisfies the equality in \eqref{kt=} if and only if $[\gamma]$ is proportional to $[\eta^{p-1}\wedge\alpha]$. 
\end{rem}

\begin{rem}
Given Theorems \ref{thm2} and \ref{thm3},  one can also use the similar arguments to extend \cite[Theorem 1.3(1)]{Li16} to our current setting. 
\end{rem}

\section*{Acknowledgements}
The author is grateful to Jian Xiao for a number of helpful discussions in Spring 2019. The author also thanks Ping Li and Jian Xiao for valuable comments on the results and pointing out several papers related to this work, and the referees for the careful reading and very useful comments and suggestions, which improve the presentation of this paper. The author particularly thanks an anonymous referee for pointing out the result in Remark \ref{general rem}.

\end{document}